\documentclass[12pt,sumlimits]{amsart}
\usepackage{hyperref}
\usepackage{amssymb}
\def\re{\mathbb{R}}
\usepackage{amsmath,amssymb,amscd}
\newcommand{\der}[3][]{\frac{\partial^{#1} #2}{\partial #3^{#1}}}

\newcommand{\gv}[1]{\mathbf{\Gamma}^{ ( #1 ) }}
\newcommand{\stack}[2]{\genfrac{}{}{0pt}{}{#1}{#2}}
\numberwithin{equation}{section}
\def\bb{\begin{equation}}
\def\ee{\end{equation}}
\def\bse{\begin{subequations}}
\def\ese{\end{subequations}}
\def\de{\delta}
\def\o{\omega}
\def\O{\Omega}
\def\w{\wedge}
\def\pa{\parallel}
\def\pe{\perp}
\def\dpa{d_\parallel}
\def\dpe{d_\perp}

\def\g{\Gamma}
\def\tg{\tilde{\g}}
\def\al{\alpha}
\def\be{\beta}
\def\ga{\gamma}
\def\P0{\mathcal{P}_0}
\def\Pm{\mathcal{P}_{>0}}
\def\lk{\Delta_{K}}
\def\sql{(-\Delta_K)}
\def\ep{\epsilon}
\def\l{\lambda}
\def\ls{\lesssim}
\def\rk{\re^{3+1}\times K}
\def\rt{\rfloor_{\der{}{t}}}
\newtheorem{defin}{Definition}[section]
\newtheorem{lemma}[defin]{Lemma}
\newtheorem{prop}[defin]{Proposition}
\newtheorem{theorem}[defin]{Theorem}
\newtheorem{corr}[defin]{Corollary}
\newtheorem{claim}[defin]{Claim}
\newtheorem*{rema}{Remark}
\author[b.Ettinger]{Boris Ettinger}
\address{Department of Mathematics \newline \indent University of California \newline \indent Berkeley, California 94720, USA.}
\email{ettinger@math.berkeley.edu}
\title[Dynamics of the three-form field]{Well-posedness of the equation for the three-form field in the eleven dimensional supergravity}

\begin{document}

\begin{abstract}
We analyze a semi-linear gauge-invariant wave equation which arises in the theory of supergravity. We prove that the Cauchy problem is well-posed globally in time for the fixed-gauge version of the equation for small compactly supported smooth data. We employ the method of Klainerman vector fields along with a finer analysis of the nonlinearity to establish an integrable decay in the energy estimate.
\end{abstract}
\maketitle 
\section{Introduction}
Let $K$ be a compact 7-dimensional Riemannian manifold. Then the product $\rk$ becomes an 11-dimensional Lorentzian manifold. For a 3-form $u$ on $\rk$ we use the Hodge star $*$ and the de Rahm differential $d$ of the product to formulate the following Cauchy problem:
\bse
\label{eq:Cauchy}
 \bb
\square_{\re^{3+1}}u-\Delta_K u= *(du\w du).
\ee
\bb 
u(0,\cdot)=u_0,u_t(0,\cdot)=u_1.
\ee
\ese
In this article, we will prove the following statement.
\begin{theorem}
\label{th:main}
There exist positive $N,\ep$ such that the Cauchy problem (\ref{eq:Cauchy}) is globally well-posed provided that initial data is localized in the ball of radius 1 in $\re^3$ for every point of $K$ and obeys
\[
 \|u_0\|_{H^N(\re^{3}\times K)}+\|u_1\|_{H^{N-1}(\re^{3}\times K)}\leq \ep.
\]
Moreover, in such a case the solution $u$ satisfies the following estimates
\[
 \sum\limits_{|\alpha|\leq N} \|\nabla_{t,x,y}\g^\alpha u(t)\|_{L^2(\re^{3}\times K)}\leq C\ep (1+t)^{1/12},
\]
\[
 \sum\limits_{|\alpha|\leq N-9 } \|\nabla_{t,x,y}\g^\alpha u(t)\|_{L^2(\re^{3}\times K)}\leq C\ep,
\]
\begin{align*}
 &(1+t)\sum\limits_{|\alpha|\leq N-18} \| \nabla_{t,x,y}\g^\alpha \P0 u(t)\|_{L^{\infty}(\re^{3}\times K)}\\
+&(1+t)^{3/2}\sum\limits_{|\alpha|\leq N-18} \|\nabla_{t,x,y}\g^\alpha \Pm u(t)\|_{L^{\infty}(\re^{3}\times K)}\leq C\ep.
\end{align*}
where $\P0,\Pm$ are the spectral projections of the operator $\lk$ defined in Section \ref{sc:hodge}, $\g^{\alpha}$ are compositions of a subset of Klainerman vector fields together with the operator $\sql^{1/2}$, which are defined in Section \ref{sc:lin}, $\nabla_{t,x,y}$ is the gradient in all the derivatives of $\re^{3+1}\times K$ and $C$ is a constant that depends only on $N$ and the geometry of $K$.
\end{theorem}
The theorem is true if the Cauchy data is supported in a larger ball but then the constant $\ep$ has to decrease as a negative power of the size of the support.
\par
The equations have a connection to the theory of supergravity, which we explain in the next section. The mathematical aspects of the supergravity theory have recently drawn the attention in the context of conformal geometry \cite{RG},\cite{Kantor}, where the space-time  was assumed to be a Riemannian manifold. The Lorentzian case was investigated earlier (see \cite{CB} and references therein).
\par
Our methods are inspired by the work of Metcalfe, Sogge and Stewart \cite{MSS} and Metcalfe and Stewart \cite{MS2} who analyze the quasi-linear wave equation on $\re^{3+1}\times D$, where $D$ is a bounded domain in $\re^n$ with various non-linearities and boundary conditions. Their results do not cover the case in study but we employ some of their ideas in this work.
\par 
The article is organized as follows. In section \ref{sc:bg} we derive the equation from a gauge-invariant Lagrangian and explain how to fix the gauge. In section \ref{sc:hodge} we recall some necessary facts from Riemannian geometry. In section \ref{sc:lin} we adapt the linear estimates for the wave equation on $\re^{3+1}$ to the product $\rk$. In section \ref{sc:nonlin} we perform a deeper analysis of the nonlinearity. In section \ref{sc:proof}, we provide the proof of Theorem \ref{th:main}.
\par We will denote by $k$ a constant which depends only on $N$ and the geometry of $K$, this constant may change from line to line but for each inequality below there is an apriori computable constant such that the inequality holds. We will also write $A \ls B$ to mean $A\leq kB$.
\section{Background}
\label{sc:bg}

\subsection{Physical Motivation}
The supergravity theory is a model of classical physics, which describes the low-energy, classical limit of the superstring theory.  The model describes the interaction of the field of gravity with other fields. In one of the simplest setups, one considers an 11-dimensional Lorentzian manifold as a space-time, with gravity field $g$ and a field, whose strength is described by a closed differential 4-form $F$. The Lagrangian is prescribed only locally by restricting the attention to an open, topologically trivial subset $U$ of the space-time. Then one solves the equation for a potential of $F$ on $U$:
\[
 dA=F.
\]
Then the Lagrangian is
\[
 \mathcal{L}=\int\limits_U Rdv+\int\limits_{U} F \w *F+ \int\limits_U A\w F \w F,
\]
where $R$, $dv$ and $*$ are the (scalar) Ricci curvature, the volume form and the Hodge $*$ corresponding to $g$, respectively. The reader should consult \cite{Witten} and textbooks for physical aspects of this theory.
\subsection{The Lagrangian and the equation}
We will simplify our setup to consider the product manifold $\rk$ as the fixed space-time, where $K$ is a 7-dimensional compact Riemannian manifold without a boundary. The metric on the product space will be the product of the Minkowski metric and the metric on $K$. We will also assume that the field strength $F$ is not only closed but also exact, namely there exists a global 3-form $u$, such that
\[
 du=F.
\]
Then $u$ will be the dynamical variable, for which we define a classical field theory Lagrangian
\bb
\mathcal{L}(u)=\int\limits_{\rk} du \w *du+\int\limits_{\rk}{ u \w du \w du}.
\ee 
The formal Euler-Lagrange equations are 
\bb
\label{eq:main1}
d*du=-du\w du. 
\ee
We will take the Hodge-dual on both sides of the equation and use the notation $\de=*d*$ to arrive to the following equation
\bb
\label{eq:main1*}
 \de du= -*(du\w du).
\ee
Since our space time is a product manifold then most of the operators which act on it can be decomposed in a natural way as operators acting on either on $\re^{3+1}$ or $K$. We will  denote by subscript $\parallel$ the operators acting on $\re^{3+1}$ and by subscript $\perp $ the operators acting on $K$. For instance, we will have
\bb
\label{eq:decompose}
d=d_{\rk}=d_{\re^{3+1}}\otimes \text{id}_K+\text{id}_{\re^{3+1}}\otimes d_{K}=\dpa\otimes \text{id}_K+\text{id}_{\re^{3+1}}\otimes \dpe.
\ee
The tensor notation should be understood in terms of operations on differential forms $\O(\re^{3+1})$ and $\O(K)$.
The equation (\ref{eq:decompose}) we will colloquially write
\bb
\label{eq:d}
d=\dpa+\dpe.
\ee

\subsection{Hodge star and form Laplacian}
\label{sbs:*}
Let us recall a few simple facts regarding the Hodge star operator, which is an operator that takes differential $n$-forms to differential $(11-n)$-forms. Let $x^i,i=0..3$ be the coordinates on $\re^{3+1}$ and $x^i,i=4..10$ be a coordinate patch on $K$ at a point where the metric tensor is the identity and its derivative vanish, i.e. normal coordinates.
Then the Hodge dual $*$ for an $n$-form $v=v_{i_1 i_2 ...i_n}dx^{i_1} dx^{i_2} ...dx^{i_n}$, $i_k=0..10$ is defined as follows:
\bb
\label{eq:dual}
*v=\sum_{i_0, i_2 ...,i_n=0}^{10}(-1)^{\alpha(\{i_0...i_n\})}\varepsilon^{i_0...i_{10}}v_{i_0 i_2 ...i_n}dx^{i_{n+1}}dx^{i_{n+2}}...dx^{i_{10}}
\ee
where 
\[
 \varepsilon^{i_0...i_{10}}=
\begin{cases}
 0, \quad i_k=i_l\text{ for some }k\neq l,\\
 1, \quad i_0..i_{10}\text{ is an even permutation},\\
 -1, \quad i_0..i_{10}\text{ is an odd permutation}
\end{cases}
\]
and 
\[
 \alpha(\{i_0...i_n\})=\begin{cases} 1,\quad 0\in \{i_0...i_n\}\\ 0,\quad 0\notin  \{i_0...i_n\}. \end{cases}.
\]
Thus the $*$ operator exchanges the components of the forms, multiplying those containing the time $x^0$ coordinate by $-1$. 
Next, we define $\de$ which takes $n$-forms to $(n-1)$-forms by
\[
 \de u=(-1)^{\text{deg} u} *d(*u).
\]
Lastly we define
\[
 \square_{\rk}=-d\de -\de d.
\]

We have the following facts about $*$ and $\de$
\begin{itemize}
 \item  $**u=(-1)^{\text{deg} u} u$.
\item $u\w *v =g(u,v)d\text{vol}$, where $g$ is the Lorentzian metric on $\rk$ and $d\text{vol}$ is the volume form.
\item The operator $*$ is an isometry and in particular $d*u=0$ if and only if $\de u=0$.
\item $\de$ is the Lorentzian adjoint of $d$ in the sense that 
\[
 \int\limits_{\rk}{ g(\de u,v)d\text{vol}}=\int\limits_{\rk}{ g(u,dv)d\text{vol}}.
\]
\item In normal coordinates and with notational conventions of relativity we have
\[
 (\de u)_{\al_1\al_2...}=-\partial^{\al_0}u_{\al_0\al_1\al_2...}.
\]
\[
 (\square_{\rk}u)_{\al_1\al_2...}=\partial^{\al_0}\partial_{\al_0}u_{\al_1\al_2...}+(f(R)u)_{\al_1\al_2...},
\]
where $f(R)$ is a linear, zeroth-order tensor that depends on the Riemann curvature tensor.
\item
$\square_{\rk}*u=*\square_{\rk}u$.
\item We have
$$\square_{\rk}du=d(\square_{\rk}u),\quad \square_{\rk}\de u=\de (\square_{\rk} u).$$
The first identity is the consequence of the fact that $d^2=0$. This fact also implies $\de^2=0$ which leads to the second identity above.
\end{itemize}
\subsection{Gauge fixing}
There is an obvious gauge freedom in equation (\ref{eq:main1*}) - if $u$ is a solution of the equation then for any two-form $w$, the three-form $u+dw$ is a solution as well. Therefore, we will fix the gauge by requiring that $u$ satisfies
\bb
\label{eq:gauge}
d*u=0,
\ee
which is equivalent to 
\[
 \de u=0.
\]
This choice of gauge is similar to the Lorenz gauge of the Maxwell equations, where for a one-form $\mathfrak{a}$ one requires $d*\mathfrak{a}=0$. Since we work with 3-forms on a product manifold, the gauge is structurally more complicated. We will give a proof that equation (\ref{eq:gauge}) is a valid gauge choice in the end of this section but first we rewrite the equation (\ref{eq:main2}) using the gauge. We defined the Laplace(d'Alembert)-Beltrami operator on forms as $\square_{\rk} u=-\de d- d \de$
\footnote{We make a consistent effort to have the operator $\Delta_K$ be negative on Riemannian manifolds. With this convention
\[
 \square_{\re^{3+1}}=\der[2]{}{t}-\sum\limits_{i=1}^3 \der[2]{}{x_i}.
\]
}. On a product manifold, the operator decomposes into $\square_{\rk}=\square_{\re^{3+1}}-\Delta_{K}$,where $\Delta_{K}$ is the Laplace-Beltrami operator on $\O(K)$ (the space of differential forms).
Thus we can rewrite the main equation (\ref{eq:main1*}) as 
\bb
\square_{\rk} u= -\de du- d\de u= *(du\w du).
\ee
Therefore, the equation (\ref{eq:main1*}) with the gauge choice (\ref{eq:gauge}) becomes
\bb
\label{eq:main2}
\square_{\re^{3+1}}u-\Delta_K u= *(du\w du).
\ee
Let us now address the validity of the gauge choice.
\begin{prop} Let $F\in \Omega^4(\rk)$. Suppose there exists a solution $A\in \Omega^{3}(\rk)$ to the equation $dA=F$. Then there exists a solution to the system
\[
d\tilde{A}=F,\quad d*\tilde{A}=0.
\]
Moreover, we can choose $\tilde{A}$ such that 
\[
\text{supp}{\tilde{A}}\subseteq\{(t,x)\times K|(s,y,z)\in \text{supp} A, |x-y|\leq |t-s| \}
\]
for $t,s\in \re,x,y\in \re^3, z\in K$. 
\end{prop}
\begin{proof}
Let $A$ be as above. Denote $\de A=e$. The two-form $e$ is the error we wish to eliminate by finding a two-form $b$ such that $\de (A+db)=0$. We solve the equation
\bb
\label{eq:corr}
-\square_{\rk} b=\de db+d\de b=e,
\ee
We have from equation (\ref{eq:corr})
\bb
\label{eq:tildeA}
 \de \tilde{A}=d\de b.
\ee
We have 
\[
 \square \de b=\de \square b=-\de e=-\de^2A=0.
\]
Thus $\de b$ solves the homogeneous wave equation. We will prove that $\de b=0$ by choosing suitable Cauchy data for $b$ at the hypersurface $t=0$. Our goal is to make the Cauchy data for $\de b$ be zero. The Cauchy data that we prescribe for $b$ in the normal coordinates are as follows: 
\bse
\bb
b_{\al_1\al_2}=0,\quad \al_1,\al_2=0..10.
\ee
\bb
 \der{}{t}b_{0\al}=0,\quad \al=1..10.
\ee
\bb
\label{eq:Cdata}
 \der{}{t}b_{\al_1\al_2}=-A_{0\al_1\al_2}=A_{\al_1 0\al_2},\quad \al_1,\al_2=0..10.
\ee
\ese
 We check the Cauchy data for $\de b$ at $t=0$.
\[
 (\de b)_{\al}|_{t=0}=-\der{}{t} b_{0\al}+\sum_{\be\neq0}\der{}{x_\be}b_{\be\al}=0.
\]
This is because the Cauchy data for $b_{0\al}$ is zero and the function $b_{\be\al}=0$ for $\al\neq 0$ at $t=0$ and so are the spatial derivatives. To see that the time derivative of $\de b$ at $t=0$ is zero, we employ the equation for $b$
\begin{align*}
\der{}{t}(\de b)_{\al}|_{t=0}=&\der[2]{b_{0\al}}{t}+\sum_{\mu\neq 0}\der{}{t}\der{}{x_\mu}b_{\mu\al}\\
=&-(\Delta_{\re^3}+\lk)b_{0\al}-(\delta A)_{0 \al}+\sum_{\mu}\der{}{t}\der{}{x_\mu}b_{\mu\nu}.
\end{align*}
Observe that $(\Delta_{\re^3}+\lk)b_{0\al}=0$ since the function at $t=0$ is zero and we take spatial derivatives and zero order term which are linear in $b$ to compute the Laplacian. The second and the third terms cancel because since $A_{00\al}=0$ by antisymmetry then
\[
 (\delta A)_{0 \al}=\sum_{\mu\neq 0}\der{}{x_\mu}A_{\mu 0\al}=\der{}{t}\sum_{\mu\neq 0}\der{}{x_\mu} b_{\mu \al},
\]
by equation (\ref{eq:Cdata}).
Therefore
\[
 (\de b)_{\al}|_{t=0}=0 
\]
and thus $\de b$ obeys a homogeneous wave equation with zero Cauchy data, which makes it identically zero. Therefore
\[
 \de \tilde{A}=\de (A+db)=d\de b=0.
\]
Observe that the support of the Cauchy data for $b$ is contained in the support for $A$ and thus the statement on the support follows from finite speed of propagation.
\end{proof}
We now wish to prove that the gauge condition $\de u=0$ persists for the equation (\ref{eq:main2}). For that we need to discuss the initial conditions. Since the equation is of second order, the natural Cauchy data is $u|_{t=0}$ and $\der{}{t}u|_{t=0}$. If we express (\ref{eq:main1*}) through the field strength $F=du$ we have
\bb
\label{eq:mainf}
 \de F=-*(F\w F).
\ee
The natural initial condition for this first order equation is $F|_{t=0}$ but we first need to observe that there is a certain compatibility condition in (\ref{eq:mainf}), which is not of the evolution form . For that we recall the notion of the interior product of a form by a vector field. Let $\al$ be an $n$-form and $X$ be a vector field, then the interior product of $\al$ by $X$, denoted by $\al\rfloor_X$ is an $n-1$ form defined by
\[
 \al\rfloor_X(X_1,X_2,..,X_{n-1})=\al(X,X_1,X_2,..X_{n-1}).
\]
We will be interested in interior products by $\der{}{t}=\der{}{x_0}$. Such a construction
can be simply described as freezing the first index of the form $\al$ to be the zero (i.e. time) index. Thus in coordinates
\[
 (\al\rfloor_{\der{}{t}})_{a_1a_2...}=\al_{0a_1a_2...}.
\]
With this notation we prove the following observation
\begin{claim}
 The form $(\de F)\rfloor_{\der{}{t}}$ does not contain the time derivatives of $F$.
\end{claim}
\begin{proof} We will give the proof in normal coordinates. We have
\[
 [(\de F)\rfloor_{\der{}{t}}]_{a_1a_2}=(\de F)_{0a_1a_2}=\der{}{x_0}F_{00a_1a_2}-\sum\limits_{i=1}^{10} \der{}{x_i}F_{i0a_1a_2}.
\]
Only the first term contains the time derivative but $F_{00a_1a_2}=0$ due to antisymmetry of $F$.
\end{proof}
Thus applying $\rfloor_{\der{}{t}}$ to (\ref{eq:mainf}) and restricting it to time $t=0$ we see that both sides of the equality
\[
 (\de F)\rfloor_{\der{}{t}}|_{t=0}=-[*(F\w F)]\rfloor_{\der{}{t}}|_{t=0}
\]
depend only on $F|_{t=0}$ so they are functions of a gauge invariant part of the Cauchy data and express a compatibility condition, which must hold in both gauge-invariant and gauge-fixed versions of the equations. We thus make the following definition.
\begin{defin} The form $du|_{t=0}$ is compatible if 
\[
 (\de du)\rfloor_{\der{}{t}}|_{t=0}=-*(du\w du)\rfloor_{\der{}{t}}|_{t=0}.
\]
\end{defin}
We now can prove that the gauge condition $\de u=0$ persists in the equation for the compatible Cauchy data.
\begin{prop}
 Let $u$ solve
\[
 \square_{\rk}u=*(du\w du),
\]
such that $\de u|_{t=0}=0$ and $du|_{t=0}$ is compatible. Then $\de u=0$ for all times.
\end{prop}
\begin{proof} We apply $\square_{\rk}$ to $\de u$ to see that
\begin{align*}
 \square_{\rk}(\de u)&=\de (\square_{\rk}) u=\de *(du\w du)=*d**(du\w du) \\
&=*d(du\w du)=2*(d^2u\w du)=0.
\end{align*}
We check the Cauchy data: $\de u|_{t=0}=0$ by assumption. The term $\der{}{t}\de u|_{t=0}$ vanishes because of the equation and the compatibility condition. We prove that in normal coordinates. We have
\bb
 [(d\de u)\rt]_{ab}=(d \de u)_{0ab}=\der{}{t}(\de u)_{ab}-\der{}{x_a}(\de u)_{0b}+\der{}{x_b}(\de u)_{0a}.
\ee
If $a,b\neq 0$ then the last two terms above are spatial derivatives of $\de u$ which is zero when we compute at $t=0$. Therefore, for $a,b\neq 0$
\begin{align*}
 \der{}{t}(\de u)_{ab}|_{t=0}&=[(d\de u)\rt]_{ab}|_{t=0}\\
&=[(d\de u)\rt]_{ab}|_{t=0}=[(-\square u -\de du)\rfloor_{\der{}{t}}]_{ab}|_{t=0}=\\
& =\lbrace[-*(du\w du) -\de du)]\rfloor_{\der{}{t}}\rbrace_{ab}|_{t=0}=0, 
\end{align*}
where we applied the compatibility condition to conclude the last equality.
Next we assume without loss of generality that $a=0,b\neq 0$ then since $\de^2 u=0$ we have
\[
 \der{}{t}(\de u)_{0b}=\de^2 u+\sum_{i\neq 0}\der{}{x_i}(\de u)_{ib}=\sum_{i\neq 0}\der{}{x_i}(\de u)_{ib}.
\]
This vanishes because it is a sum of spatial derivatives of components of $\de u$ which vanish at $t=0$.
\end{proof}
\begin{corr}
 Let $u$ solve the equation 
\[
 \square_{\rk} u=*(du\w du)
\]
with $\de u|_{t=0}=0$ and compatible $du|_{t=0}$. Then $u$ solves
\[
 \de du=-*(du\w du).
\]
\end{corr}

\section{Review of Hodge theory}
\label{sc:hodge}
The objects of our study are 3-forms on $\rk$. The basic example of such a form would be $u_\pa\w u_\pe$ where $u_\pe$ is a $k$-form (for $k=0,1,2,3$) and $u_\pa$ is a $3-k$ form on $\re^{3+1}$. The action of the Hodge-Laplacian of K is clearly $\lk (u_\pa\w u_\pe)=u_\pa\w (\lk u_\pe)$ and it extends through density on all the forms on $\rk$. Moreover, if we use the eigenvectors of $\lk$, $\lk e_\l=-\l^2 e_\l$, we can further decompose any form on $\rk$  as a series $u(x,y)=\sum_{\l}u_{\l}(x)e_{\lambda}(y)$, where $x$ is a variable on $\re^{3+1}$ and $y$ is the variable on $K$. Thus, we envision the equation being the system of equations on differential forms on $\re^{3+1}$ which are indexed by $\lambda$, in which case the equation will become
\[
\square u_{\l} -\l^2 u_\l= \sum\limits_{\l',\l''}{\mathcal{B}_\l^{\l',\l''}(u_{\l'}, u_{\l''})},
\]
where $\mathcal{B}$'s are bilinear differential operators. Thus we see that $u_\l$ for $\l=0$ evolve under a non-linear wave equation, while $u_\l$ for $\l\neq 0$ evolve under a non-linear Klein-Gordon equation. This analysis follows the ideas of Metcalfe, Sogge and Stewart \cite{MSS} and Metcalfe and Stewart \cite{MS2}, who analyze the wave equation on $\re^{n+1}\times D$, where $D$ is a bounded domain in $\re^m$ with various boundary conditions. Their analysis splits the function to eigenfunctions of the Laplacian on $D$ with appropriate boundary conditions.
\par In this section, we recall some properties of the eigenvectors of $\lk$ which we need for the proof. The material is taken from textbooks, \cite[section 2.1]{Jost}
and \cite[section 5.8]{Taylor}. For the rest of this section we will deal only with forms on $K$. We will continue to employ the subscript $\pe$ to maintain consistency. We begin with the following facts.
\begin{prop}
 \begin{enumerate}
  \item The operator $\lk$ is a differential operator acting on the space of forms $\bigoplus\limits_{i=0}^7 \O^i(K)$ with the principal symbol $g_{ij}\xi^i\xi^j \text{Id}$, where $g$ is the Riemannian metric.
\item The operator $\lk$ has a self-adjoint nonpositive-definite extension to the space of $L^2$-valued forms on $K$.
 \end{enumerate}
\end{prop}
Denote $\P0=\chi_{\{ 0\}}(-\lk),\P0=\chi_{\{\l:\l> 0\}}(-\lk)$. These are spectral projections on the zero-,non-zero subspace of the spectrum of $-\lk$, respectively.
\subsection{Hodge Theory} The range of $\P0$, i.e all the forms $\o$ that satisfy $\lk \o=0$ are called the harmonic forms.
We have the following simple fact.
\begin{claim}
\label{cl:har}
 \[
  d_{\pe}\P0=0.
 \]
\end{claim}
\begin{proof}
Let $\de_K=d_\pe^*$ be the $L^2(K)$ adjoint of $d_\pe$ 
 We have the following characterization of $\lk$ (see \cite[Definition 2.1.2]{Jost}\footnote{Our definition is the negative of \cite{Jost}})
\[
 -\lk= \de_K d_\pe+d_\pe \de_K=d_\pe ^* d_\pe +d_\pe d_\pe^*.
\]
Therefore for $\o=\P0 \o$, we have $\lk \o=0$. Thus
\begin{align*}
 0= -\langle \lk \o, \o \rangle_{L^2(K)}&= \langle d_\pe ^* d_\pe \o +d_\pe d_\pe^* \o, \o \rangle_{L^2(K)}\\
&=\| d_\pe \o\|_{L^2(K)}^2+\| d_\pe^* \o\|_{L^2(K)}^2,
\end{align*}
where $L^2(K)$ is the space of $L^2$ valued differential forms on $K$.
\end{proof}
The full version of this claim can be found in \cite[Proposition 2.1.5]{Jost}. It is the basis of Hodge theory in algebraic topology. We will not require any of it in this paper but we will quote the following theorem for the sake of beauty.
\begin{theorem}
 Every non-empty de-Rham cohomology class of $K$ contains precisely one harmonic form.
\end{theorem}
See \cite[Theorem 2.2.1]{Jost} for the proof. Thus existence and properties of harmonic forms are connected to the topological properties of the manifold. For instance the sphere $\mathbb{S}^7$ will have only two harmonic forms - the constant 0-form and the volume 7-form. The torus $\mathbb{T}^7$ will have $\binom{7}{n}$ linearly independent harmonic $n$-forms. Observe that both of these statements are independent of the choice of the Riemannian metric.
\subsection{Elliptic regularity for $\lk$}
We recall some basic regularity results for the form Laplacian. We have the following estimates
\begin{claim} Let $\o$ be a form on $K$ then
 \[
\|\o\|_{H^2(K)}\leq C (\|\Delta_K \o\|_{L^2(K)}+ \|\o\|_{L^2(K)}).
\]
\end{claim}

See \cite[Proposition 8.1]{Taylor} for proof. This inequality has the following immediate corollaries:
\begin{corr} 
\begin{enumerate}
\label{cr:elreg}
\item $\P0 L^2(K)$ is finite dimensional.
\item For every $N$ there exists a constant $C_N$ such that for every $\o\in L^2(K)$
\[
\frac{1}{C_N}\|\P0 \o\|_{H^k(K)}\leq\|\P0 \o\|_{L^{\infty}(K)}\leq C_N\|\P0 \o\|_{H^k(K)}, k\leq N, x\in \re^3.
\]
\end{enumerate}
\end{corr}

\begin{corr} for every $N$, there are constants $A_N$ such that for every $\o\in H^n(K)$
\[
\frac{1}{A_N}\|\Pm \o\|_{H^n(K)}\leq \|\sql^{\frac{n}{2}} \o\|_{L^2(K)}  \leq A_N\|\Pm \o\|_{H^n(K)},\quad \forall n\leq N.
\]
\end{corr}
See the discussion leading to \cite[Equation (8.20)]{Taylor} for the proofs. 
The practical conclusion that we will draw from these two corollaries is that when measuring smoothness of the solution in $K$ variables, we can ignore the question completely for $u= \P0 u$ and use the $\sql^{1/2}$ operator for $u=\Pm u$.
\section{Linear Estimates}
\label{sc:lin}
In this subsection, we would like to obtain decay estimates for the linear inhomogeneous equation. We will leverage this decay by employing the following subset of Klainerman vector fields:
\bb
\label{eq:poinc}
\tg=\{\partial_{i},i=0..3\}\cap \{ \O_{ij},i,j=0..3\},
\ee
where
\[
\O_{ij}=x_i\der{}{x_j}-x_j\der{}{x_i},\quad 1\geq i,j\geq 3
\]
\[
\O_{0j}=x_0\der{}{x_i}+x_i\der{}{x_0},\quad 1\geq i \geq 3
\]
We augment $\tg$ with the operator $\sql^{1/2}$
\[
\g=\tg\cup \{ \sql^{1/2} \}=\{\partial_{i},i=0..3\}\cup \{ \O_{ij},i,j=0..3\}\cup \{ \sql^{1/2} \}.
\]
We will index the set $\g$ by $i=1..11$ and for a multi-index $I=(I_1,I_2,..,I_{|I|})\in \{1,..,11\}^{|I|}$ we define the composition
\[
 \g^I=\g_{I_1}\g_{I_2}..\g_{I_{|I|}}.
\]
We will introduce some notation to simplify the presentation.  We will denote for an integer $N$, an abstract vector valued function:
\[
\gv{N}f=(\g^{\al}f)_{|\al|\leq N}.
\]
Accordingly we will interpret the following notations
\[
|\gv{N}f|=\sum_{|\al|\leq N}|\g^{\al}f|
\]
and 
\[
\|\gv{N}f\|_{p}=\sum_{|\al|\leq N}\|\g^{\al}f(t,x,y)\|_{L^p(\re^3\times K)}.
\]
We will also have a similar notation for the gradients
\[
|\nabla\gv{N}f|=\sum_{|\al|\leq N}\sum\limits_{i=0}^3|\der{}{x_i}\g^{\al}f|+\sum_{|\al|\leq N}|\sql^{\frac{1}{2}}\g^\al f|
\]
and
\begin{align*}
\|\nabla\gv{N}f\|_{p}=&\sum_{|\al|\leq N}\sum\limits_{i=0}^3\|\der{}{x_i}\g^{\al}f\|_{L^p(\re^3\times K)}\\
&+\sum_{|\al|\leq N}\|\sql^{\frac{1}{2}}\g^\al f\|_{L^p(\re^3\times K)}
\end{align*}
All those norm will be taken at a certain time $t$, which we will drop from the notation when there is no ambiguity. We will fix coordinate patches on $K$, with the appropriate partition of unity. That will turn our objects into vector valued functions on $\re^{3+1}\times \re^7$, so that we will apply the vector fields $\tg$ simply by applying them on every component of $u$.
\subsection{Linear estimates in $\re^{3+1}$} We recall the following estimates in $\re^{3+1}$ which we  seek to generalize to the product case $\rk$.
\begin{prop} Let $w\in C^{\infty}(\re^{3+1})$ such that $w(t,x)=0,t< 2B$ and $\square_{\re^{3+1}}w(t,x)=0$ for $|x|>t-B$ then
\label{pr:wavedec}
\begin{align*}
(1+t)|\nabla_{t,x}w(t,x)|\lesssim& \|\nabla_{t,x}\gv{2}w(2B,\cdot)\|_{L^2(\re^3)}\\\
&+\sum\limits_k \sup\limits_{\tau\in[2^{k-1},2^{k+1}]\cap [2B,t]}{ 2^k\|\gv{2}\square w(\tau,\cdot)\|_{L^2(\re^3)}}. 
\end{align*}
\end{prop}
This proposition is proven in \cite[Proposition 3.1]{MS2}. Although \cite{MS2} proves it with zero Cauchy data, the estimate with non-zero Cauchy data is proven in the same manner.
\begin{prop} Let $w\in C^{\infty}(\re^{3+1})$ such that $(\square_{\re^{3+1}}+1)w(t,x)=0$ for $|x|\geq t-B$ then
 \begin{align*}
(1+t)^{3/2}\sup_x(|w(t,x)|&\\
+|\nabla_{t,x}w(t,x)|) \lesssim&\|\nabla_{t,x} \gv{5} w(2B,x)\|_{L^2(\re^3)}
\\
&+ \sum_k \sup \limits_{\stack{\tau\in}{[2^{k-1},2^{k+1}]\cap [2B,t]}} 2^k\|\gv{5}F(\tau,\cdot)\|_{L^2(\re^3)},
 \end{align*}
where $F=(\square_{\re^{3+1}}+1)w$.
\end{prop}
This proposition is proved in \cite[Proposition 7.3.6]{Hor}, refining the previous work of Klainerman \cite{KKG}.
\subsection{Linear estimates in $\rk$}
We turn to obtaining estimates for the equation 
\[
\square u -\lk u= F,\quad u(0)=u_0,\dot{u}(0)=u_1.
\]
Since the spectral projections $\mathcal{P}_A$ commute with this equation, we will split the equation into two equations
\[
\square \P0 u =\P0 F
\]
and
\[
\square \Pm u -\lk \Pm u= \Pm F,
\]
with the spectral projections applied to initial data as well. By elliptic regularity and the estimate for the wave equation, Proposition \ref{pr:wavedec}, we have the following estimate.
\begin{prop}
\label{pr:0dec}
Let $\text{supp} F(\cdot,\cdot,y)\subseteq \{(t,x):|t-|x||\leq 1 \}$ for every $y\in K$ then the solution of 
\[
\square \P0 u =\P0 F
\]
obeys the estimate
\begin{align*}
(1+t)|\nabla_{t,x} (\P0 u)(t,x,y)|\lesssim& \|\nabla_{t,x}\gv{2} (\P0 u)(0,\cdot,\cdot)\|_2\\
&+\sum\limits_k \sup\limits_{s\in[2^{k-1},2^{k+1}]\cap [0,t]}{ 2^k\|\gv{2} F(s,\cdot,\cdot)\|_2}. 
\end{align*}
\end{prop}
\begin{proof}
 We wish to apply Proposition \ref{pr:wavedec} with $B=1$. For that we need to switch to a new coordinate $\tau=t+2$ then the proposition applies with one reservation: the vector fields in $\tau,x,y$ coordinates are different from the vector fields in $t,x,y$ coordinates but they are expressible in terms of sums of the old ones since $\der{}{\tau}=\der{}{t}$ and
\[
 \O_{0i}(\tau)=\tau \der{}{x_i}+x_i\der{}{\tau}=(t+2)\der{}{x_i}+x_i\der{}{y}=\O_{0i}(t)+2\der{}{t}.
\]
Thus, the Proposition \ref{pr:wavedec} applies with possibly a different constant and $(t,x,y)$ vector fields to show that for every $y\in K$
\begin{align*}
(1+t)|\nabla_{t,x} (\P0 u)(t,x,y)|\lesssim & 
\|\nabla_{t,x}(\P0 u)(0,x,y)\|_{L^2(\re^{3})}
\\
&+\sum\limits_k \sup\limits_{s\in I_k}{ 2^k\|\gv{2} (\P0 F)(s,\cdot,y)\|_{L^2(\re^{3})}},
\end{align*}
where $I_k=[2^{k-1},2^{k+1}]\cap [0,t]$.
Apply elliptic regularity (Corollary \ref{cr:elreg}) to dominate $(\P0 F)(s,\cdot,y)$ by its $L^2(K)$ norm.
\end{proof}
\begin{theorem}
\label{th:kgdec}
Let $u(t,x,y)$ solve 
\[
(\square_{\re^{3+1}}-\Delta_K)\Pm u=\Pm F,
\]
then 
\begin{align*}
(1+t)^{3/2}|\Pm u(t,x,y)|\lesssim&\|\nabla\gv{9}\Pm u(0,x,y)\|_2\\
&+ \sum\limits_k \sup\limits_{s\in[2^{k-1},2^{k+1}]\cap [0,t]}{ 2^k\|\gv{9}\Pm F(,s,\cdot)\|_2},
\end{align*}
provided $\Pm F(\cdot,\cdot,y)$ is supported in $\{(x,t)||t-|x||\leq 1 \}$ for every $y\in K$.
\end{theorem}
The proof of the theorem follows almost verbatim the proof of \cite[Proposition 7.3.5]{Hor} with the exception of the following modification of \cite[Lemma 7.3.4]{Hor}
\begin{lemma}
\label{lm:mani}
Let $K$ be a compact manifold. 
Let $u\in \O(K)$ solve the equation
\[
\der[2]{}{t}\Pm u-\Delta_K \Pm u=\Pm F,\quad u(0)=u_0,\der{u}{t}(0)=u_1,
\]
then
\begin{align*}
\|\Pm u\|_{L^{\infty}(K)}\leq& \|\lk^{2} \Pm u_0\|_{L^2(K)}+\|\sql^{3/2} \der{}{t}\Pm u_0\|_{L^2(K)}\\
&+\int_0^t{\|\Delta_K^{3/2}\Pm F(s,\cdot)\|_{L^2(K)}ds}.
\end{align*}
\end{lemma}
\begin{proof}[Proof of Lemma \ref{lm:mani}]
We combine the energy estimate for the equation
\[
 (-\der[2]{}{t}+\lk)(-\lk)^{3/2}u=(-\lk)^{3/2}F,
\]
which is
\[
 \|\lk^2 u(t)\| \ls \| \lk^2 u(0)\|+\int\limits_0^t{\|(-\lk)^{3/2}F(s)\|ds}
\]
with the Sobolev embedding for a 7-dimensional manifold:
\[
\|\Pm u(t)\|_{L^{\infty}} \ls \|\lk^2 u(t) \|_{L^2(K)}.
\]
\end{proof}
 We relegate the rest of the proof of the Theorem \ref{th:kgdec} to the appendix. We combine the two decay estimates above with the possibility to apply $\g^{\al}$ which are symmetries of the equation in the following statement.
\begin{corr}
\label{cr:dec} Let $M$ be an integer greater then 9.
Let $u$ solve 
\[
(\square_{\re^{3+1}}-\Delta_K)u=F,
\]
 with initial data concentrated in the ball of radius 1 for every $y\in K$ and 
\[ \text{supp} F(\cdot,\cdot,y)\subseteq \{(t,x):|t-|x||\leq 1 \}\]
 for every $y\in K$ then the following estimate holds:
\bb
\label{eq:decaygen}
\begin{split}
(&1+t)\|\nabla\gv{M-9}\P0 u(t)\|_{\infty}\\
&\begin{split}+(1+t)^{3/2}\|\nabla\gv{M-9}\Pm u(t)\|_{\infty}\ls
 \|&\nabla\gv{M}u(0)\|_2\\
&+\sum\limits_k \sup\limits_{s\in I_k}{ 2^k\|\gv{M} F(s,\cdot)\|_{2}},
\end{split}
\end{split}
\ee
where $I_k=[2^{k-1},2^{k+1}]\cap [0,t]$.
\end{corr}
\subsection{Energy estimates} We combine the energy estimates for the solution of $\square_{\rk} u=F$  with the fact that the operators $\g$ are symmetries of the equation and use the notation introduced above. 
\begin{prop}
Let $u$ be the solution of $\square_{\rk} u=F$ then for any $M\geq 0$ we have
 \bb
  \|\nabla \gv{M} u(t)\|_2\leq\|\nabla 
\gv{M} u(0)\|_2+\int\limits_0^t {\|\gv{M}F(s)\|_{L^2}ds},
\ee
\end{prop}
for any $M\geq 0$.

\section{Analysis of Nonlinearity}
\label{sc:nonlin}
In this section, we will treat the bilinear form $(u_1,u_2)\mapsto *(du_1 \w du_2)$. 
Recall from Subsection \ref{sbs:*} the $*$ operator exchanges the components of the forms, multiplying those containing the time $x^0$ coordinate by $-1$. The $*$ operator loses the simple form when the metric on $K$ is no longer the identity, but because of tensoriality, it will be multiplied by a function depending only on $x^4,..x^{10}$, which due to compactness will be bounded above and below. Therefore, when we take $L^2(\re^{3}\times K)$-norm at a certain time, we will consider $*v$ to be $L^2$ equivalent to $v$. Furthermore, we will be interested in the action of $\g$ operators on $*(du\w du)$. Clearly, the operators which act on $\re^{3+1}$ componentwise will commute with $*$. The equation (\ref{eq:dual}) shows that the Laplacian $\lk$ commutes with $*$ simply because $\lk=\sum\limits_{i=4}^{10}\der[2]{}{x_i}$ at that point and the relation is tensorial. Thus any function of $\lk$ will commute with $*$ and we have 
\[
 \sql^{1/2} *v=*\sql^{1/2}v.
\]
From this discussion we conclude that
\bb
\|\g^{\alpha} (*v(t))\|_{L^2(\re^{3}\times K)} \cong\|*\g^{\alpha} v\|_{L^2(\re^{3}\times K)}\cong\|\g^{\alpha} v\|_{L^2(\re^{3}\times K)},
\ee
for any multi-index $\alpha$, time $t$, with constants which depend only on the manifold $K$.
\subsection{The splitting of the nonlinearity} Recall that the operator $d$ splits into $d=d_\pa+d_\pe$. Also any form $u$ on $\re^{3+1}\times K$ can be written as $u=\P0 u+ \Pm u$. Therefore, we can write
\bb
\begin{split}
 *(du \w du )=& *[ (d_\pa+d_\pe)(\P0 u+ \Pm u)\w (d_\pa+d_\pe) (\P0 u+ \Pm u))]
\\ =&* (d_\pa \P0 u \w d_\pa \P0 u)\\
&+2*(d_\pa \P0 u \w d \Pm u)+*(d \Pm u \w d \Pm u).
\end{split}
\ee
where we used that $d_\pe \P0=0$, which is the content of Claim \ref{cl:har}. With this we proved the following splitting of the nonlinearity:
\begin{claim} 
\label{cl:split}
Let $u_1,u_2$ be differential 3-forms on $\re^{3+1}\times K$.
 Denote
\[
 B(\P0 u_1,\P0 u_2)=*(d_\pa \P0 u_1\w d_\pa \P0 u_2),
\]
\[
 C(\P0 u_1,\Pm u_2)=*(d_\pa \P0 u_1\w d \Pm u_2),
\]
\[
 D(\Pm u_1,\Pm u_2)=*(d \Pm u \w d \Pm u).
\]
Then 
\[
 *(du\w du)=B(\P0 u,\P0 u)+2C(\P0 u,\Pm u)+D(\Pm u,\Pm u).
\]

\end{claim}
\subsection{The basic estimate}
\begin{prop}
\label{pr:base}
 Let $F$ be any of the forms $B,C,D$ defined in Claim \ref{cl:split} or the total nonlinearity which is $B+2C+D$. Let $N$ be a positive integer. Then there exists a constant $k$ such that for any $M\leq N$ we have
\begin{align*}
 \|\gv{M}F(v_1,v_2)\|_2\ls&\ \|\nabla\gv{\frac{M}{2}}v_1\|_{p_1}\|\nabla\gv{M}v_2\|_{q_1}\\
&+\|\nabla\gv{M}v_1\|_{q_2}\|\nabla\gv{\frac{M}{2}}v_2\|_{p_2},
\end{align*}
where $\frac{1}{p_i}+\frac{1}{q_i}=\frac{1}{2}$, $2\leq p_i,q_i \leq \infty$.
\end{prop}
\begin{proof}
 Choose a coordinate patch $x^i$,$i=4..10$ for $K$. Then de Rham differentials $d,d_\pe,d_\pa$ can be written as $a_i\der{}{x_i}$ for $a_i$ which depend only on $x^i,i\geq 4$. Thus we need to estimate an expression of the form
\[
 I=\sum_{i,j=0}^{10}\g^\al (a_ia'_j\der{}{x_i}v_1 \der{}{x_j}v_2),
\]
 Observe that all the operators in $\g$ besides $\sql^{1/2}$ are vector fields and thus obey Leibniz's rule. So assume first that in the composite operator $\g^{\al}$ there are no $\sql^{1/2}$ operators. We treat the Hodge dual $*$ as a constant coefficient operator, which only permutes between different components. Next, we employ the Jacobi identity to write $I$ as
\[
 \sum_{\al'+\al''=\al}C_{\al'\al''}\sum_{i,j=0}^{10}a_ia'_j\g^{\al'} (\der{}{x_i}v_1) \g^{\al''}(\der{}{x_j}v_2),
\]
where $C_{\al'\al''}$ are constants. Observe that $\g_i$ commutes with $\der{}{x_i}$, for $i\geq 4$ and for $i,j\leq 4$ we have
\[
 [\der{}{x_k},\O_{0j}]=\delta_{0k}\der{}{x_j}+\delta_{jk}\der{}{x_0},[\der{}{x_k},\O_{ij}]=\delta_{jk}\der{}{x_k}-\delta_{ik}\der{}{x_j}.
\]
Thus we commute $\der{}{x_i}$ with $\g$'s to get
\[
 I=\sum_{|\be'|+|\be''|\leq M}\sum_{i,j=0}^{10}C_{\be'\be''ij}(\der{}{x_i}\g^{\be'}  v_1) (\der{}{x_j}\g^{\be''}v_2),
\]
for some constants $C_{\be'\be''ij}$. In the expression above only one of the $|\beta'|,|\beta''|$ can be larger then $M/2$. We split the sum accordingly
\begin{align*}
 |I|\lesssim& \sum_{|\be'|\leq \frac{M}{2},i}|\der{}{x_i}\g^{\be'}  v_1|\sum_{|\be''|\leq M,j} |\der{}{x_j}\g^{\be''}v_2|\\
&+\sum_{|\be'|\leq M,i}|\der{}{x_i}\g^{\al'}  v_1| \sum_{|\be''|\leq \frac{M}{2},j}|\der{}{x_j}\g^{\be''}v_2|
\end{align*}
We then obtain the required estimate by applying the $L^2$ norm to $|I|$ and using the appropriate H\"older inequalities. 
\par In case $\g^{\al}$ contains $k$ $\sql^{1/2}$ operators, we note that $\sql^{1/2}$ commutes with all the other $\g$ operators as the rest of $\g$ operate on $\re^{3+1}$ only. By elliptic regularity
\[
\|I\|=\|\sql^{\frac{k}{2}} \g^{\alpha'}F\|\ls \|\Pm \g^{\alpha'}F\|_{H^k(K)}\leq \|\g^{\alpha'}F\|_{H^k(K)}
\]
\end{proof}
Now $H^k(K)$ norm obeys a Jacobi's ``inequality``, which is a primitive form of the Kato-Ponce estimates, see 
\cite{KP},\cite[Chapter II, Prop. 1.1]{Taylor2} 
and thus the rest of the proof proceeds in a similar fashion. 
\subsection{Null form}
Continuing with the notation of Claim \ref{cl:split}, we need the observation that $B(\P0 u_1,\P0 u_2)=*(d_\pa \P0 u_1 \w d_\pa \P0 u_2)$ is a null form and the estimates that follow from it.
\begin{prop} The bilinear form $B(\o_1,\o_2)=*(\dpa \o_1 \w \dpa \o_2)$ is a null-form
\label{pr:null}
\end{prop} 
We will give two proofs of the proposition.
\begin{proof}[Fourier] It is enough to compute $B(\o_1,\o_2)$ for $\o_i=A_ie^{ik_ix}$ for two parallel null-vectors, with $A_i$ being constant. If $B(\o_1,\o_2)$ vanishes in such a case then $B$ is a null-form. But
\[
B(\o_1,\o_2)=*(k_1\w A_1 \w k_2 \w A_2)e^{i(k_1+k_2)x}.
\]
When two vectors in the wedge product are parallel - the wedge product vanishes. Thus
\[
B(\o_1,\o_2)=0.
\]
\end{proof}
\begin{proof}[Coefficients]
Let
\[
A=A_{ijk}dx_i dx_j dx_k
\]
and
\[
B=B_{lmn}dx_l  dx_m  dx_n.
\]
Then
\[
\dpa A=\der{A_{ijk}}{x_p}dx_p  dx_i dx_j dx_k,
\]
\[
\dpa B=\der{B_{lmn}}{x_s}dx_s dx_l dx_m dx_n .
\]
Therefore
\[
\dpa A\w \dpa B=\der{A_{ijk}}{x_p}\der{B_{lmn}}{x_s}dx_p  dx_i dx_j dx_k  dx_s  dx_l dx_m dx_n\]
\[ 
\quad
+\der{A_{ijk}}{x_s}\der{B_{lmn}}{x_p}dx_s dx_i  dx_j dx_k dx_p dx_l dx_m dx_n
\] 
\[
=(\der{A_{ijk}}{x_p}\der{B_{lmn}}{x_s}-\der{A_{ijk}}{x_s}\der{B_{lmn}}{x_p})dx_p  dx_i dx_j dx_k  dx_s  dx_l dx_m dx_n.
\]
Thus
\begin{align*}
*\dpa A \w \dpa B=&(\der{A_{ijk}}{x_p}\der{B_{lmn}}{x_s}-\der{A_{ijk}}{x_s}\der{B_{lmn}}{x_p})\\
&*(dx_p  dx_i dx_j dx_k  dx_s  dx_l dx_m dx_n).
\end{align*}
Since
\[
\der{A_{ijk}}{x_p}\der{B_{lmn}}{x_s}-\der{A_{ijk}}{x_s}\der{B_{lmn}}{x_p}
\]
is a null-form, $*\dpa A \w \dpa B$ is a null-form. Observe that we needed to assume that only $dx_p$ and $dx_s$ are co-vectors on $\re^{3+1}$; all the other indices could have belonged to either $K$ or $\re^{3+1}$. In order to see that the sign in front of the second term is negative, we count the transpositions needed to transform $$\o_1=dx_s dx_i  dx_j dx_k dx_p dx_l dx_m dx_n$$ to $$\o_2=dx_p  dx_i dx_j dx_k  dx_s  dx_l dx_m dx_n$$ one needs four to bring $dx_p$ to the front and then another three to bring $dx_s$ behind $dx_i dx_j  dx_k$ therefore there are 7 transpositions in total and the sign is minus, $\o_1=-\o_2$.
\end{proof}
Denote
\bb
Q_{ij}(f,g)=\der{f}{x_i}\der{g}{x_j}-\der{f}{x_j}\der{g}{x_i}.
\ee
As the proof of Proposition \ref{pr:null} shows the form $B(\o_1,\o_2)=*(\dpa \o_1\w \dpa \o_2)$ is the sum of the forms $Q_{ij}$ applied to different components. 
We have the following estimate
\begin{lemma}
\label{lm:null}
\[
|Q_{ij}(f,g)|\leq C \frac{1}{1+t}\left(|\g f||\nabla_{x,t} g|+ |\nabla_{x,t} f| |\g g|. \right)
\]
\end{lemma}
\begin{rema} The proof is taken from \cite[Lemma 1.1]{Kla1}. We reproduce it here to stress that we have the estimate involving only vector fields $\g$ and not the full set of Klainerman vector fields, which includes the radial scaling field $x_0\der{}{x_0}+..+x_3\der{}{x_3}$.
\end{rema}
\begin{proof}
We have
\[
\der{}{x_i}=-\frac{x_i}{t}\der{}{x_0}+\frac{1}{t}\O_{0i},\quad i=1,2,3,
\]
Thus we have for $i,j\geq 1$
\[
Q_{ij}(f,g)=\frac{1}{t}[-\der{f}{x_0}\O_{ij}g+(\O_{0i}f \der{g}{x_j}-\O_{0j}f \der{g}{x_j})].
\]
For $i=0$, we have
\[
Q_{0j}=\frac{1}{t}(\der{f}{x_0}\O_{0j}g-\O_{0j}f\der{g}{x_0}).
\]
\end{proof}
We wish to prove a variant of the basic estimate specialized to the null form. 
\begin{prop}
\label{pr:basenull}
\[
 \begin{split}
 \|\gv{M}B(v,v)\|_2\lesssim&\ \frac{1}{1+t}\Big(\|\gv{\frac{M}{2}+1}v\|_{p_1}\|\nabla\gv{M}v\|_{q_1}\\
&+\|\gv{M+1}v\|_{q_2}\|\nabla\gv{\frac{M}{2}}v\|_{p_2}\Big),
\end{split}
\]
\end{prop}
\begin{proof}
The bilinear form $B$ is a linear combination of the forms $Q_{ij}$. The forms $Q_{ij}$ are preserved under applications of the $\g$ operators (with the exception of $\sql^{1/2}$) because of the following formula
which appears in \cite[Lemma 1.2]{Kla1} and which can be obtained by direct calculation
\bb
\label{eq:dernull}
\O_{\al\be}Q_{\ga\de}(f,g)=Q_{\ga\de}(\O_{\al\be}f,g)+Q_{\ga\de}(f,\O_{\al\be} g)+\tilde{Q}(f,g), 
\ee
where
\[
\tilde{Q}(f,g)=m_{\al\ga}Q_{\be\de}(f,g)-m_{\be\ga}Q_{\al\de}(f,g)+m_{\al\de}Q_{\be\ga}(f,g)+m_{\be\de}Q_{\al\ga}(f,g),
\]
 where $m_{\al\be}$ are coefficients of the Minkowski metric. Thus we apply Lemma \ref{lm:null} and after grouping the multi-indices with order less then $\frac{M}{2}$ and applying the H\"older estimates we get the result like in Proposition \ref{pr:base} . For $\sql^{1/2}$ we apply Kato-Ponce estimates.
\end{proof}

\section{Proof of Theorem \ref{th:main}}
\label{sc:proof}
 To prove the theorem we seek to establish the following apriori bounds for the solutions of (\ref{eq:Cauchy}):
\bse
\label{eqs:as}
\bb
\label{eq:HE}
(1+t)^{-\delta}\|\nabla \gv{N} u(t) \|_2 \leq \ep ,
\ee
\bb
\label{eq:LE}
\|\nabla \gv{N-9} u(t) \|_2 \leq \ep, 
\ee
\bb
\label{eq:decay}
(1+t)\|\nabla \gv{N-18} \P0 u(t) \|_\infty +(1+t)^{3/2}\|\nabla \gv{N-18} \Pm u (t)\|_\infty  \leq \ep,
\ee
\ese
where $\delta$ is smaller then $\frac{1}{12}$.

Any of the global in time estimates above implies uniqueness, existence and well-posedness for the semilinear wave equations by employing the local theory which is explained in \cite[Chapter 2]{Sogb} or \cite[section 6.2]{Hor}. We will prove the estimates by bootstrapping. Namely, we will prove that (\ref{eqs:as}) imply
\bse
\bb
\label{eq:HEh}
(1+t)^{-\delta}\|\nabla \gv{N} u(t) \|_2 \leq \frac{\ep}{2} ,
\ee
\bb
\label{eq:LEh}
\|\nabla \gv{N-9} u(t) \|_2 \leq \frac{\ep}{2}, 
\ee
\bb
\label{eq:decayh}
(1+t)\|\nabla \gv{N-18} \P0 u(t) \|_\infty +(1+t)^{3/2}\|\nabla \gv{N-18} \Pm u(t) \|_\infty  \leq \frac{\ep}{2},
\ee
\ese
i.e. the right-hand side can be made $\frac{\ep}{2}$ instead  of $\ep$.
Thus our goal is to establish the following statement.
\begin{prop}
\label{pr:main}
Let $N$ be an integer that satisfies
\bb
\frac{N}{2}\leq N-18.
\ee
Then there exists $\ep>0$ small enough such that, for every $T>0$, if
a solution $u$ to the Cauchy problem (\ref{eq:Cauchy}) with initial data that satisfies 
\bb
\|\gv{N}u_0\|_2+\|\gv{N-1}u_1\|\leq \frac{\ep}{4},
\ee
such that $u_0(\cdot,y),u_1(\cdot,y)$ are localized in a ball of radius 1 for every $y\in K$ and if the inequalities (\ref{eq:HE}),(\ref{eq:LE}), (\ref{eq:decay}) hold for every $0<t\leq T$ then the inequalities (\ref{eq:HEh}),(\ref{eq:LEh}), (\ref{eq:decayh}) hold for every $0<t\leq T$,
where $\delta$ is a positive exponent which depends on $\ep$ and is smaller then $\frac{1}{12}$

\end{prop}
\begin{rema} There are two considerations that affect the smallness of $\ep$. One is that we will see that we can replace $\ep$ in the right-hand side of (\ref{eqs:as}) by $\frac{\ep}{4}+k\ep^2$, where $k$ is will an apriori computable constant that depends only on $N$ and the geometry of the manifold $K$. Thus we will need to decrease $\ep$ to achieve the inequality 
\[
\frac{\ep}{4}+k\ep^2\leq \frac{\ep}{2}.
\]
Another consideration is that the exponent $\delta$ depends linearly on $\ep$, $\delta=C\ep$ with $C$ depending on $N$ and the geometry of $K$. We have to decrease $\ep$ so that $\delta=C\ep\leq \frac{1}{12}$.
\end{rema}
We prove the implication (\ref{eq:HEh}) in the next lemma, while the implications (\ref{eq:LEh}),(\ref{eq:decayh}) are proved in Lemma \ref{lm:lowh}.
\begin{lemma} 
\label{lm:HEh}
Under conditions of Proposition \ref{pr:main}, (\ref{eq:decay}) implies (\ref{eq:HEh}).
\end{lemma}
\begin{proof}
We use the energy estimate for the equation $\square \gv{N}u=\gv{N}(*du\w du)$.
\bb
\label{eq:eeh}
\|\nabla\gv{N} u (t)\|_2\leq \|\gv{N}\nabla u(0)\|_2+\int\limits_0^t \|\gv{N}*du\w du(s)\|_2ds.
\ee
We now use the Proposition \ref{pr:base} to estimate the nonlinearity. We have
\[
 \|\gv{N}*du\w du(s)\|_2\leq C\|\nabla \gv{\frac{N}{2}}u\|_{\infty}\|\nabla \gv{N} u\|_2.
\]
Since $\frac{N}{2}\leq N-18$ we employ the assumption (\ref{eq:decay}) in  equation (\ref{eq:eeh}) to conclude
\[
\|\nabla\gv{N} u (t)\|_2\leq \frac{\ep}{4}+\int\limits_0^t \frac{C\ep}{1+s}\|\nabla\gv{N}u(s)\|_2ds.
\]
After applying Gronwall inequality we conclude
\[
\|\nabla\gv{N} u (t)\|_2\leq \frac{\ep}{4}(1+t)^{C\ep},
\]
which is the required inequality.
\end{proof}

We require the following interpolated intermediate result.
\begin{claim} Let $2\leq p \leq \infty$ then:
 \begin{enumerate}
  \item Assumptions (\ref{eq:HE}) and (\ref{eq:decay}) imply
\bb
\label{eq:interpH}
(1+t)^{-\delta+1-\frac{2}{p}}\|\nabla\gv{N-9}\P0 u \|_{p}+(1+t)^{-\delta+3/2(1-\frac{2}{p})}\|\nabla\gv{N-9}\Pm u \|_{p}\ls \ep.
\ee
  \item Assumptions (\ref{eq:LE}) and (\ref{eq:decay}) imply
\bb
\label{eq:interp}
(1+t)^{-1+\frac{2}{p}}\|\nabla\gv{N-18}\P0 u \|_{p}+(1+t)^{-3/2(1-\frac{2}{p})}\|\nabla\gv{N-18}\Pm u \|_{p}\ls \ep.
\ee
 \end{enumerate}
\end{claim}
\begin{proof}
 Obviously, the estimate for sum with more derivatives is true for the sum with less derivatives and thus we will interpolate between equation (\ref{eq:LE}) and equation (\ref{eq:decay}) to get the second conclusion. To prove the first point, we  use  the following intermediate result,
\bb
\label{eq:interm}
(1+t)^{-\delta+1}\|\nabla\gv{N-9}\P0 u \|_{\infty}+(1+t)^{-\delta+3/2}\|\nabla\gv{N-9}\Pm u \|_{\infty}\ls \ep.
\ee
By interpolation, of equation (\ref{eq:interm}) with equation (\ref{eq:HE}) we have
\[
(1+t)^{-\delta+1-\frac{2}{p}}\|\nabla\gv{N-9}\P0 u \|_{p}+(1+t)^{-\delta+3/2(1-\frac{2}{p})}\|\nabla\gv{N-9}\Pm u \|_{p}\ls \ep.
\]
To establish (\ref{eq:interm}) we use equation (\ref{eq:decaygen})
\begin{align*}
(1&+t)\|\nabla\gv{N-9}\P0 u \|_{\infty}\\
&\begin{split}+(1+t)^{3/2}\|\nabla\gv{N}\Pm u \|_{\infty}\leq&\|\gv{N-9}u(0)\|_2
\\&+\sum\limits_n \sup\limits_{\tau\in I_n}{ 2^n\|\gv{N}F (\tau))\|_{2}},
\end{split}
\end{align*}
where $I_n=[2^{n-1},2^{n+1}]\cap [0,t]$ and $F=*(du\w du)$. We use Proposition \ref{pr:base} and combine it with the assumptions to get
\begin{align*}
\|\gv{N}F(\tau)\|_2= \|\gv{N}*du\w du(\tau)\|_2&\ls C\|\nabla \gv{\frac{N}{2}}u\|_{\infty}\|\nabla \gv{N} u\|_2\\
&\ls \ep^2(1+\tau)^{\delta-1}.
\end{align*}
Therefore,
\begin{align*}
(1&+t)\|\nabla\gv{N-9}\P0 u \|_{\infty}\\
&
\begin{split} +(1+t)^{3/2}\|\nabla\gv{N-9}\Pm u\|_{\infty}
 \ls&\|\gv{N-9}u(0)\|_2\\
 &+ \sum\limits_n \sup\limits_{\tau\in[2^{n-1},2^{n+1}]\cap [0,t]}\ep  2^n \tau^{-1+\delta}\\
\leq& \frac{\ep}{4}+k \ep^2 (1+t)^{\delta},
\end{split}
\end{align*}
which proves (\ref{eq:interm}).
\end{proof}

To complete the proof of Proposition \ref{pr:main}, we analyze the equation 
\[
\square_{\rk}\gv{N-9}u=\gv{N-9}*du\w du.
\]
We estimate the right-hand side of the equation in the following lemma.
\begin{lemma} 
\label{lm:F}
Under assumptions (\ref{eq:HE}),(\ref{eq:LE}),(\ref{eq:decay}) we have
\[
\|\gv{N-9}*du\w du(t)\|_2\ls \ep^2(1+t)^{-1-\frac{1}{3}+2\delta}.
\]
\end{lemma}
\begin{proof}
Recall the splitting of the nonlinearity in Claim \ref{cl:split} and denote
\bse
\label{eq:forms}
\bb
 B(t)=\|\gv{N-9}B(\P0 u(t), \P0 u(t))\|_2,
\ee
\bb
C(t)=\|\gv{N-9}C(\P0 u(t), \Pm u(t))\|_2,
\ee
\bb
D(t)=\|\gv{N-9}D(\Pm u(t), \Pm u(t))\|_2.
\ee
\ese
To prove the lemma, we will obtain the bound above for $B(t)$,$C(t)$,$D(t)$.
\begin{description}
\item[Estimate for $B(t)$]
To estimate $B(t)$ we use Proposition \ref{pr:basenull} to get
\begin{align}
\label{eq:B1}
 \|\gv{N-9} B(\P0 u,\P0 u)\|_2  \leq& \frac{k}{1+t}\bigl(\|\gv{\frac{N}{2}+1}\P0 u \|_{6}\|\nabla \gv{N-9} \P0 u\|_{3} \\ \nonumber
&+\|\gv{N-8}\P0 u \|_{6}\|\nabla \gv{\frac{N}{2}} \P0 u\|_{3} \bigr).
\end{align}
We wish to use the homogeneous Sobolev embedding in $\re^3$. Thus, we have for every $y\in K$,
\[
 \|\gv{M}\P0 u(t,\cdot,y) \|_{L^6(\re^3)}\leq\|\nabla\gv{M}\P0 u(t,\cdot,y) \|_{L^2(\re^3)}.
\]
Next we employ the compactness of $K$
to see that  the $L^6(K)$ norm of $\P0 \gv{M} u$ is dominated by the $L^\infty(K)$ norm. At the same time, elliptic regularity shows that the $L^\infty(K)$ norm of $\P0\nabla \gv{M}u$ is dominated by its' $L^{2}(K)$ norm. Therefore
\bb
\label{eq:B2}
\|\gv{\frac{N}{2}+1}\P0 u \|_{6}\leq \|\gv{N-8}\P0 u \|_{6} \leq \|\nabla\gv{N-8}u\|_{2} \leq k\ep(1+t)^{\de},
\ee
by (\ref{eq:interpH}) and the fact that $\frac{N}{2}\leq N-18$.
Employing (\ref{eq:interpH}) again we have
\bb
\label{eq:B3}
\|\nabla\gv{N-9} u\|_{6}\ls \ep(1+t)^{\delta-\frac{1}{3}}.
\ee
Also equation (\ref{eq:interp}) implies
\bb
\label{eq:B4}
\|\nabla \gv{\frac{N}{2}} \P0 u\|_{3}\leq \|\nabla\gv{N-18} \P0 u\|_{3}\ls \ep (1+t)^{-\frac{1}{3}}.
\ee
Thus, applying (\ref{eq:B2}),(\ref{eq:B3}),(\ref{eq:B4}) on (\ref{eq:B1})  we have
\bb
\label{eq:B}
B(t)=\|\gv{N-9} B(\P0 u,\P0 u)\|_2\leq k\ep^2 (1+t)^{2\delta -\frac{1}{3}-1}.
\ee
\item[Estimate for $C(t)$]
For $C(t)$ we have the following estimate.
We use Proposition \ref{pr:base} to see that
\begin{align*}
\| \gv{N-9} C(\P0 u, \Pm u)\|_2 \ls&
 \ \|\nabla \gv{N-9} \P0 u\|_{2}\|\nabla \gv{\frac{N}{2}} \Pm u\|_{\infty}\\
 &+\|\nabla \gv{\frac{N}{2}} \P0 u\|_{3}\|\nabla \gv{N-9} \Pm u\|_{6}.
\end{align*}
By equation (\ref{eq:interp})
\[
\|\nabla\gv{\frac{N}{2}} \P0 u\|_{3}\leq \ep (1+t)^{-\frac{1}{3}}
\]
and by (\ref{eq:interpH})
\[
\|\nabla \gv{N-9} \Pm u\|_6 \leq \ep (1+t)^{\delta-1}.
\]
For the second summand, we will use the bootstrap assumptions (\ref{eq:LE}) and (\ref{eq:decay}).
Thus we get
\bb
\label{eq:C}
C(t)=\|\gv{N-9} C(\P0 u, \Pm u)\|_2\leq k\ep^2 ((1+t)^{\delta-\frac{4}{3}}+(1+t)^{-\frac{3}{2}}).
\ee
\item[Estimate for $D(t)$]
Lastly, we estimate $D(t)$. We have
\[
\|\gv{N-9}D(\Pm u, \Pm u)\|_2\leq k(\|\nabla \gv{\frac{N}{2}} \Pm u\|_{\infty}\|\nabla \gv{N-9} \Pm u\|_2),
\]
which according to the bootstrap assumptions satisfies
\bb
\label{eq:D}
D(t)=\|\gv{N-9}D(\Pm u, \Pm u)\|_2\leq k \ep^2 \frac{1}{(1+t)^{\frac32}}.
\ee

\end{description}
We combine (\ref{eq:B}),(\ref{eq:C}),(\ref{eq:D}) to obtain the required estimate.
\end{proof}
We are now ready to complete the proof of the Proposition \ref{pr:main} in the following lemma.
\begin{lemma}
\label{lm:lowh} Under the conditions of Proposition \ref{pr:main}, the assumptions (\ref{eq:HE}),(\ref{eq:LE}),(\ref{eq:decay}) imply 
(\ref{eq:LEh}),(\ref{eq:decayh})
\end{lemma}
\begin{proof} We analyze the equation 
\bb
\label{eq:dw}
\square_{\rk}\gv{N-9}u=\gv{N-9}(*du\w du).
\ee
First apply the energy estimate to get
\begin{align*}
\|\gv{N-9} u(t)\|_2&\leq \|\gv{N-9} u(t)\|_2+\int\limits_0^t{\|\gv{N-9}(*du\w du)(s)\|_2 ds}\\
&\leq \frac{\ep}{4}+k\ep^2\int\limits_0^t \frac{1}{(1+t)^{1+1/12}}ds
\leq \frac{\ep}{4}+k\ep^2\leq \frac{\ep}{2},
\end{align*}
where we used Lemma \ref{lm:F} and the assumptions. This establishes (\ref{eq:LEh}). To address (\ref{eq:decayh}), apply (\ref{cr:dec}) to equation (\ref{eq:dw}) to get
\begin{align*}
(1&+t) \|\nabla\gv{N-18}\P0 u(t)\|_{\infty}\\
&
\begin{split}
+(1+t)^{3/2}\|\nabla\gv{N-18}\Pm u(t)&\|_{\infty}\\
\leq&\ \|\nabla \gv{N-9}\P0 u(0)\|_2\\
&+\|\sum\limits_n \sup\limits_{s\in I_n}{ 2^n \gv{N-9}(F(s)))}\|_2, 
\end{split}
\end{align*}
where  $I_n=[2^{n-1},2^{n+1}]\cap [0,t]$ and $F=*du\w du$. Since $$\|\gv{N-9}F(s)\|_2\leq \ep^2(1+s)^{-1-\frac{1}{12}}$$  by Lemma \ref{lm:F}, we have
\begin{align*}
(1&+t)\|\nabla\gv{N-18}\P0 u(t)\|_{\infty}\\
&
\begin{split}
+(1+t)^{3/2}\|\nabla\gv{N-18}\Pm u(t)\|_{\infty}
&\leq \frac{\ep}{4}+k\sum\limits_{n\leq C\log t+1}  2^n \ep^2 2^{-n(1+\frac{1}{12})}\\
&\leq \frac{\ep}{4}+k\ep^2,
\end{split}
\end{align*}
which establishes (\ref{eq:decayh}) and completes the proof.

\end{proof}

\section{Acknowledgements} The author would like to thank his doctoral adviser Daniel Tataru for suggesting the problem and several key ideas in the proof, Jason Metcalfe for discussing the technical aspects of \cite{MSS} and \cite{MS2}, Robin Graham for exposing the author to questions of supergravity and explaining the results of conformal geometry, Jon Aytac for pointing to the work of Choquet-Bruhat \cite{CB} and Tobias Schottdorf for helping improve the exposition. 
\appendix
\section{Proof of the theorem \ref{th:kgdec}}
We will denote $\tau=t+2$. We will again employ operators $\g$ in $\tau$ variable, which are different from the operators in $t$ variable but can be expressed as a linear combination with coefficients independent of $u$. See proof of Proposition \ref{pr:0dec} for details regarding this substitution. 
We will use the following statement
\begin{prop}
\label{pr:hyper}
Let $g$ be supported in $\{(\tau,x);T/2\leq \tau \leq T,|x|\leq \tau\}$. Denote 
\[
 M(\rho)=\sup_{\tau^2-|x|^2=\rho^2} |g(\tau,x)|
\]
then 
\[
 T^{2}\int M(\rho^2)^2\rho d\rho\leq C\sum\limits_{|I|\leq \frac{5}{2}}\int {|\tg^I g(\tau,x)|^2d\tau dx}.
\]
\end{prop}
The proof of this statement is given in \cite[Lemma 7.3.1]{Hor}.
Let $u$ solve
\[
 \square u -\lk u=\der[2]{u}{\tau}-\sum\limits_{i=1}{3}\der[2]{u}{x_i}-\lk u= F.
\]
Introduce the hyperbolic polar coordinates $(\tau,x)=\rho\o,\rho=(\tau^2-|x|^2)^{\frac{1}{2}},\o\in \mathbb{S}^2$. Then the equation becomes
\[
 \der[2]{u}{\rho}+\frac{3}{\rho}\der{u}{\rho}-\lk u=\frac{1}{\rho^2}\Delta_H u +f,
\]
 where $\Delta_H$ is the Laplacian on the hyperbolic space. We have
\[
 \Delta_H=\sum_{i} (t\der{}{x_i}+x_i\der{}{t})^2-\sum_{k,j}(x_j\der{}{x_k}-x_k\der{}{x_j})^2
\]
Thus $v=\rho^\frac{3}{2} u$ obeys
\[
 \der[2]{v}{\rho}-\lk v=\rho^{3/2}(\rho^{-2}(\Delta_H u+ 3\frac{u}{4})+F).
\]
We decompose the right-hand-side dyadically in time. Let $\chi$ be a smooth function, supported on $[\frac{1}{2},2]$ s.t. $\sum\limits_{k=-\infty}^\infty{ \chi(\frac{\tau}{2^k})}=1$. Denote $f_k=\chi(\frac{\tau}{2^k})F$, $u_k=\chi(\frac{\tau}{2^k})(\Delta_H u+ n(n-2)\frac{u}{4})$
We apply the Lemma \ref{lm:mani}. We have 
\bb
\label{eq:sumk}
 |v(\rho,\o,y)|\leq \sum_k\int\limits_0^\rho {\rho^{3/2}(\rho^{-2}\|u_k(\rho,\o,\cdot)\|_{H^4(K)}+\|f_k(\rho,\o,\cdot)\|_{H^4(K)})}
\ee
We wish to estimate $u(\overline{\tau},\overline{x})$, thus we need to estimate the sum of the integrals above. Denote $M_k(\rho^2)=\sup\limits_{\o}\|f_k(\rho,\o,\cdot)\|_{H^4(K)}$ then we have
\[
 \int_0^{\overline{\rho}}{\rho^{\frac{3}{2}}\|f_k\|_{H^4(K)}d\rho}\leq \int_0^{\overline{\rho}}{M_k(\rho)^2\rho d\rho} \int_0^{\overline{\rho}}{\rho^2d\rho}.
\]
We have $2^{k-1}\leq t \leq 2^{k}$ therefore $\frac{\rho}{\overline{\rho}}\leq C\frac{\tau}{\overline{\tau}} \leq C\frac{2^k}{\overline{\tau}}$ in the support of $u$. Thus
\[
 \int \rho^2 d\rho \leq C\frac{\overline{\rho}^3 2^{3k}}{\overline{\tau}^3}
\]
According to Proposition \ref{pr:hyper}, we have 
\begin{align*}
 \int {M_k(\rho^2)\rho d\rho}\leq\ & 2^{-2k}C\sum\limits_{|I|\leq \frac{5}{2}}\int {|\tg^I \|f_k(\tau,x,\cdot)\||^2d\tau dx}\\
\leq\  &2^{-2k}C\sum\limits_{|I|\leq \frac{5}{2}}\int\limits_{\tau\leq\overline{\tau}}\int\limits_{K} {|\tg^I \lk^ 2f_k(\tau,x,y)|^2 d\tau dx dy}
\end{align*}
 Therefore, we have
\bb
\label{eq:estforce}
\begin{split}
\int {\rho^{3/2}}\|f_k(\rho,\o,\cdot)&\|_{H^4(K)}\\
\leq &\ \frac{2^\frac{k}{2}\rho^{3/2}}{\overline{\tau}^{3/2}}\sum\limits_{|I|\leq \frac{5}{2}}\int\limits_{2^{k-1}\leq \tau\leq{2^{k}}}\int\limits_{K} {|\tg^I \lk^2f_k(\tau,x,y)|^2 d\tau dx dy} \\ \nonumber
 \leq &\ \rho^{3/2}2^{k}\frac{1}{\overline{\tau}^{3/2}}\sum\limits_{|I|\leq 7}\sup\limits_{\tau\in [2^{k-1},2^{k+1}]\cap [2B,\overline{\tau}]} {\|\g^I f(\tau,\cdot,\cdot)\|}. \nonumber
\end{split}
\ee
To deal with the first term on the right-hand side of equation (\ref{eq:sumk}), we repeat the argument in \cite[Lemma 7.3.4]{Hor} verbatim, to get the estimate
\[
 \int {\rho^{-\frac{1}{2}}\|u_k\|_{H^4}}\leq \frac{\overline{\rho}^{3/2}}{\overline{\tau}^{3/2}}\sum_{|I|\leq \frac{9}{2}}\sup\limits_{2B\leq t} \|\tg^I u(t,\cdot)\|_{H^4(K)}.
\]
We use the energy estimate to bound this expression by the estimate (\ref{eq:estforce}) and the initial energy to get after summation in $k$
\[
 \overline{\rho}^{\frac{3}{2}}|u|=|v|\leq \frac{\overline{\rho}^{\frac32}}{\overline{\tau}^{\frac32}}\sum_{|I|\leq 6}{(\sum_i \|\g^I\g_i u_0\|+\sum_{k}2^k\sup\limits_{2^{k-1}\leq \tau \leq 2^{k+1}}\|\g^I f(\tau)\|)}
\]
Dividing by $\frac{\overline{\rho}^{\frac32}}{\overline{\tau}^{\frac32}}=\frac{\overline{\rho}^{\frac32}}{(t+1)^{\frac32}}$ completes the proof.

\end{document}